\ifx \processToSlides \undefined
  \ifx\usepackage\RequirePackage	
    \let\usingAmsArtXII\usepackage	
  \fi
\else
  \documentclass[12pt]{amsart}
  \usepackage{amssymb,amscd}
  \usepackage[paperwidth=11truein,paperheight=8.25truein,margin=1mm,includeheadfoot]{geometry}
  \def \useHugeSize {}
  \def \numberingIsThrough {}
\fi

\ifx \labelsONmargin\undefined

  \ifx\RequirePackage\undefined
    \expandafter\ifx\csname amsart.sty\endcsname\relax
	\ERRORlatexTwoZeroNineShouldNotBeTriggered
    \fi

    \def\mathbb{\Bbb}
    \def\mathfrak{\frak}
    \def\mathbf{\bold}
    \ifx\boldsymbol\undefined	
      \def\boldsymbol#1{{\bold #1}}
    \fi

    \def\mathbit{\boldsymbol}

    \makeatletter
    \newenvironment{proof}{%
         \@ifnextchar[{%
                       \expandafter\let\expandafter\end@proof
                         \csname endpf*\endcsname
                         \my@proof
                      }{\let\end@proof\endpf\pf}%
        }{\end@proof}
    \def\my@proof[#1]{\@nameuse{pf*}{#1}}
    \makeatother

    \def\xrightarrow[#1]#2{@>{#2}>{#1}>}
    \def\xleftarrow[#1]#2{@<{#2}<{#1}<}

    \def\providecommand#1{\def#1}
    \def\emph#1{{\em #1}}
    \def\textbf#1{{\bf #1}}

    \def\mathring{\overset{\,\,{}_\circ}}
  \else
      \ifx\usepackage\RequirePackage
      \else	
        \documentclass[12pt]{amsart}		
        \usepackage{amssymb,amscd}
	\let\usingAmsArtXII\usepackage

      \fi
      \ifx \mathring\undefined		
        \DeclareMathAccent{\mathring}{\mathalpha}{operators}{"17}
      \fi

      \long\def\FAKEendPROOF{\endtrivlist}
      \ifx\endproof\FAKEendPROOF
	  \def\endproof{\qed\endtrivlist}
      \fi

      \ifx\mathbit\undefined	
        \DeclareMathAlphabet{\mathbit}{OML}{cmm}{b}{it}
      \fi

      \ifx \usingAmsArtXII \undefined
      \else				
	\advance\oddsidemargin -0.1\textwidth
	\evensidemargin\oddsidemargin
      \fi

      \def\Sb#1\endSb{_{\substack{#1}}}
      \def\Sp#1\endSp{^{\substack{#1}}}

  \fi
\makeatletter
\@ifundefined{DeclareFontShape} 
     {\@ifundefined{selectfont}
             {
                \message{We need AmS-LaTeX, this version of LaTeX does not 
                        support it}\@@end
             }{
                \def\mathcal{\cal}
                
                \def\pcyr{%
                        \def\default@family{UWCyr}%
                        \let\oldSl@\sl
                        \def\sl{\def\default@shape{it}\oldSl@}%
                        \cyracc
                        \language\Russian\family{UWCyr}\selectfont
                }
             }}{
                \DeclareFontEncoding{OT2}{}{\noaccents@}
                \DeclareFontSubstitution{OT2}{cmr}{m}{n}
                \DeclareFontFamily{OT2}{cmr}{\hyphenchar\font45 }
                \DeclareFontShape{OT2}{cmr}{m}{n}{%
                     <5><6><7><8><9><10>gen*wncyr %
                     <10.95><12><14.4><17.28><20.74><24.88> wncyr10 %
                }{}
                \DeclareFontShape{OT2}{cmr}{m}{it}{%
                     <5><6><7><8><9><10> gen * wncyi%
                     <10.95><12><14.4><17.28><20.74><24.88> wncyi10%
                }{}
                \DeclareFontShape{OT2}{cmr}{bx}{n}{%
                     <5><6><7><8><9><10> gen * wncyb%
                     <10.95><12><14.4><17.28><20.74><24.88> wncyb10%
                }{}
                \DeclareFontShape{OT2}{cmr}{m}{sl}{%
                     <-> ssub * cmr/m/it%
                }{}
                \DeclareFontShape{OT2}{cmr}{m}{sc}{%
                     <5><6><7><8><9><10>%
                     <10.95><12><14.4><17.28><20.74><24.88> wncysc10%
                }{}
                \DeclareFontFamily{OT2}{cmss}{\hyphenchar\font45 }
                \DeclareFontShape{OT2}{cmss}{m}{n}{%
                     <8><9><10> gen * wncyss%
                     <10.95><12><14.4><17.28><20.74><24.88> wncyss10%
                }{}
                
                \def\cyrencodingdefault{OT2}
                \def\pcyr{%
                        \cyracc
                        \let\encodingdefault\cyrencodingdefault
                        \language\Russian\fontencoding{OT2}\selectfont
                }
             }   

\@ifundefined{theorembodyfont} 
     {
        \def\theorembodyfont#1{\relax}
        \makeatletter
          \let\@@th@plain\th@plain
          \def\th@plain{ \@@th@plain \slshape }
        \makeatother
        \let\normalshape\relax
     }{}   

\ifx\cprime\undefined
     \def\cprime{$'$}
\fi
\makeatletter
  \def\@sect@my#1#2#3#4#5#6[#7]#8{%
\ifnum #2>\c@secnumdepth
   \let\@svsec\@empty
 \else
   \refstepcounter{#1}%
\edef\@svsec{\ifnum#2<\@m
             \@ifundefined{#1name}{}{\csname #1name\endcsname\ }\fi
\noexpand\rom{\csname the#1\endcsname.}\enspace}\fi
 \@tempskipa #5\relax
 \ifdim \@tempskipa>\z@ 
   \begingroup #6\relax
   \@hangfrom{\hskip #3\relax\@svsec}{\interlinepenalty\@M #8\par}%
   \endgroup
   \if@article\else\csname #1mark\endcsname{%
        \ifnum \c@secnumdepth >#2\relax\csname the#1\endcsname. \fi#7}\fi
\ifnum#2>\@m \else
       \let\@tempf\\ \def\\{\protect\\}\addcontentsline{toc}{#1}%
{\ifnum #2>\c@secnumdepth \else
             \protect\numberline{%
               \ifnum#2<\@m
               \@ifundefined{#1name}{}{\csname #1name\endcsname\ }\fi
               \csname the#1\endcsname.}\fi
           #8}\let\\\@tempf
     \fi
 \else
  \def\@svsechd{#6\hskip #3\@svsec
    \@ifnotempty{#8}{\ignorespaces#8\unskip
       \ifnum\spacefactor<1001.\fi}%
        \ifnum#2>\@m \else
          \let\@tempf\\ \def\\{\protect\\}\addcontentsline{toc}{#1}%
            {\ifnum #2>\c@secnumdepth \else
              \protect\numberline{%
                \ifnum#2<\@m
                \@ifundefined{#1name}{}{\csname #1name\endcsname\ }\fi
                \csname the#1\endcsname.}\fi
             #8}\let\\\@tempf\fi}%
 \fi
\@xsect{#5}}
  \ifx\RequirePackage\undefined
  \let\@sect\@sect@my             
  \fi
  \ifx\RequirePackage\undefined	
  \def\th@remark@my{\theorempreskipamount6\p@\@plus6\p@
    \theorempostskipamount\theorempreskipamount
    \def\theorem@headerfont{\it}\normalshape}
    \let\th@remark\th@remark@my
  \else				
    \let\o@@remark\th@remark

    \ifx\theorempostskipamount\undefined		
    \else						
      \def\th@remark{\o@@remark
	\ifdim\theorempostskipamount < 2pt\relax
	  \theorempostskipamount\theorempreskipamount
	     \multiply\theorempostskipamount\tw@
	     \divide\theorempostskipamount\thr@@
	\fi
      }
    \fi
  \fi
\let\myLabel\@gobble
\def\labelsONmargin{\@mparswitchfalse\def\myLabel##1{\@bsphack\marginpar
                                  {\normalshape\tiny\rm Label ##1}\@esphack}}
\ifx \url\undefined
  \def\url#1{{\tt #1}}%
\fi
\def\PREpmodSKIP{\allowbreak  \if@display\mkern18mu\else\mkern8mu\fi}

\def\cyracc{\def\u##1{
                \if \i##1\char"1A%
                \else \if I##1\char"12%
                \else \accent"24 ##1\fi\fi }%
\def\"##1{\if e##1{\char"1B}%
                \else \if E##1{\char"13}%
                \else \accent"7F ##1\fi\fi }%
\def\9##1{\if##1z\char"19 
\else\if##1Z\char"11 
\else\if##1E\char"03 
\else\if##1e\char"0B 
\else\if##1u\char"18 
\else\if##1U\char"10 
\else\if##1A\char"17 
\else\if##1a\char"1F 
\else\if##1p\char"7E 
\else\if##1P\char"5E 
\else\if##1Q\char"5F 
\else\if##1q\char"7F 
\else\if##1i\char"1A 
\else\if##1I\char"12 
\else\if##1N\char"7D 
\fi
\fi
\fi
\fi
\fi
\fi
\fi
\fi
\fi
\fi
\fi
\fi
\fi
\fi
\fi
}%
\def\cydot{{\kern0pt}}}%
\def\cydot{$\cdot$}

\ifx\Russian\undefined
        \def\Russian{0\relax
    \message{Don't know the hyphenation rules for Russian^^J
                        Please do INITeX with `input  russhyph' in the 
                        command line}%
                \gdef\Russian{0\relax}%
        }
\fi
\ifx\RequirePackage\undefined			
  \def\@putname#1#2#3#4{\def\@@ref{#3}\let\old@bf\bf
        \def\bf##1{\old@bf\if?\noexpand##1?{#4}\else##1\fi}%
	#1{#2}%
        \let\bf\old@bf}
\else						
  \def\@putname#1#2#3#4{\def\@@ref{#3}\let\old@bf\bf	
	\let\old@reset@font\reset@font			
        \def\bf##1{\old@bf\if?\noexpand##1?{#4}\else##1\fi}%
	\def\reset@font##1##2{\old@reset@font##1\if?\noexpand##2?{#4}\else##2\fi}#1{#2}%
        \let\bf\old@bf\let\reset@font\old@reset@font}
\fi
\let\my@ref=\ref
\def\ref#1{\@putname\my@ref{#1}{#1}{\tiny\rm\@@ref}}
\let\my@pageref=\pageref
\def\pageref#1{\@putname\my@pageref{#1}{#1}{\tiny\rm\@@ref}}
\let\my@cite=\cite
\def\cite#1{\@putname\my@cite{#1}{\@citeb}{\tiny\rm\@@ref}}
\makeatother
\fi
\relax 

\ifx \SKIPstatementDEFS \undefined
  \theoremstyle{plain} 
  \theorembodyfont{\sl}
\fi

\ifx\useDoubleSpacing\undefined
\else
  \usepackage{setspace}
\fi

\ifx \address \undefined
  \if \institute \undefined \else	
     \def\address{\institute}
     \ifx \email \undefined
        \let\email\texttt
     \fi
  \fi
\fi

\let\emphOrig\emph
\ifx \doEmphToIndex \undefined
  
\else
  \usepackage{makeidx}
  \makeindex

  \def\eatToBar#1|{}
  \def\emphToIndexSLASH#1\/{\index{#1}\eatToBar}
  \def\emphToIndexDOTSLASH#1.\/{\emphToIndexSLASH #1\/}
  \def\emphAndIndex#1{\emphOrig{#1}{\emphToIndexDOTSLASH #1.\/|}}
  \let\emph\emphAndIndex
\fi

\ifx \SKIPstatementDEFS \undefined

\ifx \numberingIsThrough \undefined
\numberwithin{equation}{section}
\fi

\theorembodyfont{\rm}
\theoremstyle{definition}

\ifx \numberingIsThrough \undefined
\newtheorem{definition}{Definition}[section]
\else
\newtheorem{definition}{Definition}
\fi

\newtheorem{example}[definition]{Example}

\theoremstyle{remark}

\newtheorem{remark}[definition]{Remark} 

\ifx \numberingIsThrough \undefined
\newtheorem{note}{Note}[section] 
\newtheorem{summary}{Summary}[section] 
\else

\fi

\theoremstyle{plain} 
\theorembodyfont{\sl}

\newtheorem{theorem}[definition]{Theorem}
\newtheorem{lemma}[definition]{Lemma}
\newtheorem{corollary}[definition]{Corollary}
\newtheorem{proposition}[definition]{Proposition}

\fi 

\newcommand{\triv}{{\mathbb C}}

\newcommand{\Ind}{\operatorname{Ind}}
\newcommand{\Res}{\operatorname{Res}}
\newcommand{\str}{\operatorname{str}}

\newcommand{\id}{\operatorname{Id}}
\newcommand{\ev}{\operatorname{ev}}
\newcommand{\coev}{\operatorname{coev}}

\newcommand{\Ext}{\operatorname{Ext}}
\newcommand{\rk}{\operatorname{rk}}
\newcommand{\End}{\operatorname{End}}
\def\supp{\operatorname{Supp}}

\def\Hom{\operatorname{Hom}}
\def\Ind{\operatorname{Ind}}

\def\Rep{\operatorname{Rep}}
\def\St{\operatorname{St}}
\def\sdim{\operatorname{sdim}}
\def\Res{\operatorname{Res}}

\def\Y{\mathbb Y}

\def\s{{\bf s}}

\def\cI{{\mathcal {I}}}

\def\C{{\mathbb C}}

\def\g{{\mathfrak {g}}}

\def\q{{\mathfrak {q}}}

\def\k{{\mathfrak {k}}}
\def\l{{\mathfrak {l}}}

\def\p{{\mathfrak {p}}}
\def\s{{\mathfrak {s}}}

\usepackage{xcolor}

\ifx\useHugeSize\undefined
\else
\Huge
\fi

\relax

\begin{document}
\title{Towards the Green correspondence for supergroups}

\author{ Inna Entova-Aizenbud, Vera Serganova, Alex Sherman}

\begin{abstract}
    We prove a version of Green's correspondence for complex algebraic supergroups, constructing a correspondence between certain indecomposable representations of $G$ and the normalizer of a Sylow subgroup of $G$.
\end{abstract}

\date{ \today }

\maketitle

\section{Introduction}

The subject of this paper is the category of representations of a quasireductive supergroup $G$ over the field of complex numbers. Recall that quasireductive supergroups are supergroups with linearly reductive underlying algebraic groups (up to connected component). They are characterized by the property that the tensor category $\Rep G$ of finite-dimensional $G$-modules is Frobenius.

Other well-studied examples of such categories come from modular representation theory of finite groups. One can try to apply methods of representation theory of finite groups to representations of quasireductive supergroups.

There are already several results in this direction. Boe, Kujawa and Nakano developed the theory of cohomological support for $\Rep G$ in \cite{BoKN1}, \cite{BoKN2} and \cite{BoKN3}.
In \cite{SSV} we develop the theory of Sylow subgroups in $G$, which are analogues of Sylow $p$-groups in modular representation theory. In particular, we prove that all Sylow subgroups are conjugate.

The goal of this paper is to study the relation between representations of $G$ and the normalizer $K$ of a Sylow subgroup of $G$.  This is directly inspired by local representation theory from the modular setting.  Our first result is Theorem \ref{thm:mainss}, in which we establish a correspondence between indecomposable representations of $G$ and $K$
with non-zero superdimension. Although this correspondence is not bijective, it is one-to-one under a mild technical condition. This condition holds for all simple quasireductive supergroups except of type $Q(n)$.  In the language of tensor categories, this correspondence induces a full functor from the semisimplification of $\Rep G$ to the semisimplification of $\Rep K$.  

Theorem \ref{thm:mainss} may be viewed as the first step towards a Green's correspondence for representations of supergroups.  Our assumptions on the superdimension mean we only consider modules whose defect subgroup is the whole Sylow subgroup.  For an analogous version of Green's correspondence for finite groups, see Sec.~4 of \cite{EO}.

In the modular representation theory for a finite group $G$, it is known that if a Sylow $p$-subgroup $P$ satisfies the trivial intersection property (meaning $P\cap gPg^{-1}=P$ or $\{e\})$, then restriction induces an equivalence of stable categories $\St G\to\St N_G(P)$ (see Chpt.~10 of \cite{A}).  In the super case, the analogous trivial intersection property is too strong as it immediately implies that $P$ is normal!  However a possible weakening of the condition is as follows: if $P$ is a Sylow subgroup, then $P_0$ satisfies the trivial intersection property, where $P_0$ is the maximal even subgroup of $P$.  This is natural in light of Corollary 9.10, of \cite{SSV} which states that two Sylow $p$-subgroups $P,P'$ are equal if and only $P_0=P'_0$.

The only (nontrivial) cases in which $P_0$ satisfies the trivial intersection property is when $G$ is a defect one basic classical supergroup.  In this case, we prove that the stable categories of $G$ and $K/\Gamma$ are equivalent for some finite subgroup $\Gamma\subset K$, see Theorem \ref{st-def1}. We use this to explicitly compute the semisimplification
of $\Rep G$ at the end of Section 3. 

In the last section, we discuss a conjectural candidate for the defect group of a block for classical supergroups $GL(m|n)$ and $OSp(m|2n)$ in terms of relative projectivity. 

\subsection{Acknowledgements} I.E. and V.S. were partially supported by the NSF-BSF grant 2019694. A.S. was partially supported by ARC grant DP210100251 and an AMS-Simons travel grant.

\section{Splitting subgroups and semisimplification}

\subsection{Splitting subgroups}  Let $G$ be a quasireductive supergroup; by this we mean that the connected component of the identity of $G_0$ is connected.  In particular, $G$ need not be connected.  We will write $\mathfrak{g}$ for the Lie superalgebra of $G$.

Recall that a quasireductive subgroup $K\subseteq G$ is called splitting if the natural $G$-map $\triv \to \Ind_K^G\triv $ coming from the adjunction has a splitting $\Ind_K^G\triv \to \triv$.  If $K\subseteq G$ is splitting and $H\subseteq K$ is splitting, then $H\subseteq G$ is also splitting.

\begin{example}
We now provide the main examples of splitting subgroups.
\begin{enumerate}
    \item If $G=GL(m|n)$, then $K=GL(1|1)\times GL(m-1|n-1)$ is splitting.  In fact, $GL(1|1)^{\min(m,n)}$ is splitting in $GL(m|n)$.
    \item If $G=OSp(m|2n)$ and $r=\min(\lfloor m/2\rfloor,n)$, then $K=OSp(2|2)^r\subseteq OSp(m|2n)$ is splitting.  
    \item Suppose more generally that $\g$ is a Kac-Moody Lie superalgebra.  Recall that the defect of $\g$ is the maximal size of a linearly independent set of isotropic roots of $\g$.  If $\g$ is of defect $d$, let $S=\{\alpha_1,\dots,\alpha_d\}$ be a linearly independet set of isotropic roots.  We may consider the Lie subalgebra $\mathfrak{k}$ of $\g$ generated by all root space $\g_{\pm\alpha_1},\dots,\g_{\pm\alpha_d}$.  Then $\k\cong \s\l(1|1)^d$.  If we let $K\subseteq G$ denote the subgroup of $\k$ which integrates $\k$, then $K$ is splitting inside of $G$.
    \item From example (3), we see that if $G$ is defect 1 (e.g.~$AG(1|2),AB(1|3)$, or $D(2,1;t)$) then $SL(1|1)\subseteq G$ is splitting, where $SL(1|1)$ is the root subgroup of any isotropic root.
    \item If $G=P(n)$, then $K=P(n-2)\times P(2)$ is splitting in $G$.  In particular $P(2)^{n}$ is spltting inside of $P(2n)$, and $P(2)^n\times P(1)$ is splitting inside of $P(2n+1)$.
    \item If $G=Q(n)$, then $K=Q(n-2)\times Q(2)$ is splitting in $G$.  In particular, $Q(2)^n$ is splitting inside of $Q(2n)$, and $Q(2)^n\times Q(1)$ is splitting inside of $Q(2n+1)$.  
\end{enumerate} 

\end{example}

\subsection{Semisimplification}  Let $\mathcal C$ be a symmetric tensor category. By this we mean that $\mathcal C$ is a $\Bbbk$-linear, symmetric monoidal category such that $\mathcal{C}$ is locally finite abelian, $\End(\mathbf{1})=\Bbbk$, and every object of $\mathcal{C}$ is rigid.  Given a morphism $f:X\to X$ in $\mathcal{C}$, we define $\operatorname{tr}(f)\in\Bbbk$ to be the following morphism:
\[
\mathbf{1}\xrightarrow{\text{coev}}X\otimes X^*\xrightarrow{f\otimes 1}X\otimes X^*\cong X^*\otimes X\xrightarrow{\text{ev}}\mathbf{1}.
\]
The tensor ideal $\mathcal{N}$ of negligible morphisms is given by all morphisms $f:X\to Y$ such that $\operatorname{tr}(fg)=0$ for all morphisms $g:Y\to X$.  

\begin{definition}
    For any tensor category $\mathcal C$ we define the semisimplification of $\mathcal C$ to be the quotient category  $\overline{\mathcal C}=\mathcal{C}/\mathcal{N}$.  We write $S:\mathcal C\to\overline{\mathcal C}$ for the quotient functor; we call $S$ the semisimplification functor.

\end{definition}

When $\mathcal{C}=\Rep G$, we will write $\str$ for the categorical trace.

\subsection{Construction of the functor $F$}
  
\begin{proposition}\label{prop:splitting}
Given a splitting subgroup $K \subset G$, the functor $\Res^G_K$ takes negligible morphisms to negligible morphisms.
\end{proposition}
\begin{proof} The splitting of $\triv \to \Ind_K^G\triv $ induces the splitting of $i_M:M\to \Ind_K^G \Res_K  M$ for any $G$-module $M$ via the canonical isomorphism
  $\Ind_K^G M\simeq (\Ind_K^G\triv)\otimes M$. Let $f\in \Hom_K(\Res_K M,\Res_K N)$, define  $\tilde f:\in \Hom_G(M,N)$ as a composition
  $$M\xrightarrow{i_M}\Ind^G_KM\xrightarrow{\Ind (f)}\Ind^G_KN\xrightarrow{p_N}N,$$
  where $p_N$ is the splitting morphism and $i_M$ the canonical embedding.

  First we claim that if $f\in \Hom_G(M,N)$ then $\tilde f=f$. Indeed,
  $$i_N\circ f= \Ind(f)\circ i_M$$
  and the statement follows since $p_N\circ i_N$ is the identity.

  Next we claim that if $f\in \Hom_G(M,N)$ and $g\in \Hom_K(\Res_K N,\Res_K L)$ then $\widetilde{g\circ f}=\tilde {g}\circ f$ since both are defined by
  the compositions
  $$M\xrightarrow{f}N\xrightarrow{i_N}\Ind^G_KN\xrightarrow{\Ind(g)}\Ind^G_KL\xrightarrow{p_L}L.$$

  Finally we claim that for $f\in \End_K(\Res_K M)$ we have $\str f=\str\tilde f$. It follows from the following commutative diagram
  $$\begin{CD}\triv@>\varphi>>M\otimes M^*@>\ev>>\triv\\
    @VViV @VViV@VViV\\
    \Ind^G_K\triv@>\Ind(\varphi)>>\Ind^G_K(M\otimes M^*)@>\Ind(\ev)>> \Ind^G_K\triv\\
    @VVpV @VVpV@VVpV\\
\triv@>\tilde \varphi>>M\otimes M^*@>\ev>>\triv
  \end{CD}$$
  where $\varphi=(f\otimes\operatorname{id})\circ \coev$.

  Assume now that $f\in\Hom_G(M,N)$ is a negligible morphism and $\Res_K f$ is not negligible. Then there exists  $g\in \Hom_K(\Res_K N,\Res_K M)$ such that
  $\str (g\circ f)\neq 0$. But then $\str(\tilde g\circ f)\neq 0$ and we obtain a contradiction.

\end{proof}

\begin{remark}
    In fact, Proposition \ref{prop:splitting} can be strengthened to an if and only if statement.  Suppose that $K$ is not splitting $G$, meaning that $\C\to\Ind_{K}^{G}\C$ does not split.  Let $V\subseteq\Ind_{K}^{G}\C$ be a nontrivial indecomposable submodule with $\C\subseteq V$.  Then the map $f:\C\to V$ is negligible since it is not split.  However upon restriction to $K$ there is a splitting coming from the evaluation morphism $V\subseteq\Ind_{K}^{G}\C\xrightarrow{\text{ev}_{eK}}\C$.  Thus $\operatorname{Res}_{K}(f)$ is not negligible.
\end{remark}

\begin{theorem}\label{thrm:split-ss}
 The functor $Res^G_K: \Rep(G) \to\ Rep(K)$ descends to a tensor functor $$F:\overline{\Rep(G)}\to\overline{\Rep(K)}$$ between the semisimplification categories, so that
the diagram
    $$\begin{CD}Rep (G)@>\Res>> Rep (K)\\
      @VVSV@VVSV\\
      \overline{Rep(G)}@>F>>\overline{Rep(K)}
      \end{CD}
    $$
   is commutative. 
\end{theorem}

\begin{proof}
 The existence of a functor $F$ making the above diagram commutative follows from Proposition \ref{prop:splitting}. 
 Being a $\C$-linear functor between semisimple categories, $F$ is automatically exact; since $F$ is also symmetric monoidal, it must be faithful.
\end{proof}

\subsection{DS functor  and fullness of $F$ } Recall that $x\in\g_{\bar 1}$ is called {\it homological}
if $[x,x]$ is a semisimple element of $\g_{\bar 0}$. By $\g_{\bar 1}^{hom}$ we denote the set of all homological elements.

For any $x\in\g_{\bar 1}^{hom}$ there exists a symmetric monoidal functor $DS_x:\Rep G\to\operatorname{sVec}$ defined by 
\[
	DS_xM=\frac{\operatorname{Ker}(x:M\to M)}{\operatorname{Im}(x:M\to M)\cap\operatorname{Ker}(x:M\to M)}.
	\]
       Observe that $DS_xM$ admits a natural action of the Lie superalgebra $\g_x:=DS_x\g$.
        \begin{definition} Let $G$ be a quasireducitve supergroup and $K\subset G$ is a splitting subgroup.  A pair $(G,K)$ is called {\it proper} if for some $x\in\k_{\bar 1}^{hom}$ we have $DS_x\Ind^G_K\triv=\triv$.
          \end{definition}
Then the following Lemma is straightforward.
    
 \begin{lemma}\label{lem:ind} Let $(G,K)$ be a proper pair. Then for any $G$-module $M$ we have $DS_x\Ind^{G}_KM\simeq DS_x M$.
    \end{lemma}
    \begin{lemma}\label{lem:split} Suppose that $K$ is a splitting subgroup of $G$. Let $M$ be an indecomposable $G$-module and $M_0$ be a $K$-submodule of $\Res_K M$ which splits as a direct summand. If $\sdim M_0\neq 0$, then $M$ is a direct summand of $\Ind^G_K M_0$.
  \end{lemma}
  \begin{proof} 
    The composition of $K$-module maps $$ \triv\xrightarrow{\coev} M_0\otimes M^*_0 \to M_0\otimes M^*\xrightarrow{s} M^*\otimes M_0\to M^*_0\otimes M_0\xrightarrow{ev}\triv$$
  is $\sdim(M_0)\id_{\triv} \neq 0$ so $\triv$ is a direct summand of the $K$-module $ M_0\otimes M^*$. Hence $\Ind^G_K\triv$ is a direct summand of $\Ind ^G_KM_0\otimes M^*$. Since $K$ is a splitting subgroup, $\triv$ is a direct summand of the $G$-module $\Ind^G_K\triv$ and thus a direct summand of $\Ind ^G_KM_0 \otimes M^*$.

    This implies the splitting of $M$ in $\Ind ^G_KM_0$ by the lemma below for the case $U=\Ind ^G_KM_0$.
  \end{proof}
  \begin{lemma}\label{general} Let $U$ and $M$ be $G$-modules and $M$ be indecomposable and finite-dimensional. Assume that there are morphisms
    $$\triv\xrightarrow{\alpha}U\otimes M^*\xrightarrow{\gamma}\triv,$$
    such that $\gamma\alpha=1$. Define
    $$\bar\alpha: M\xrightarrow{\alpha\otimes 1_M}U\otimes M^*\otimes M\xrightarrow{1_U\otimes\ev} U $$
    and
    $$\bar\gamma: U\xrightarrow{1_U\otimes\coev} U\otimes M\otimes M^*\xrightarrow{1_U\otimes s_{23}}U\otimes M^*\otimes M\xrightarrow{\gamma\otimes 1_M} M$$
    Then $\bar\gamma\bar\alpha: M\to M$ is an isomorphism.
  \end{lemma}
  \begin{proof} Choose a basis $\{m_i\}$ in $M$ and the dual basis $\{\varphi_i\}$ in $M^*$. Let $\alpha(1)=\sum_i u_i\otimes \varphi_i$. Then we have $\sum_i\gamma(u_i,\varphi_i)=1$.
    We have that $\bar\alpha(m_i)=u_i$, and $\bar\gamma(u_i)=\sum_j(-1)^{m_j}\gamma(u_i,\varphi_j)m_j$. Then
    $$\operatorname{str}\bar\gamma\bar\alpha=\sum_i\gamma(u_i,\varphi_i)=1.$$
    Note that $M$ is indecomposable and hence every endomorphism of $M$ is either nilpotent or an isomorphism. Hence $\bar\gamma\bar\alpha$ is an isomorphism.
    \end{proof}
    \begin{proposition}\label{prop:full} Let $(G,K)$ be some proper pair and
     $M$ be an indecomposable $G$-module with $\sdim M\neq 0$. Then $\Res_K M$ has exactly one indecomposable component of non-zero superdimension.
    \end{proposition}
    \begin{proof} Assume the opposite. Then we can write $\Res_K M=M_1\oplus M_2$ where both $M_1$ and $M_2$ have non-zero superdimension. Then by Lemma \ref{lem:ind},
      $$DS_x M=DS_x\Ind^G_KM=DS_x\Ind^G_KM_1\oplus DS_x\Ind^G_KM_2.$$
      On the other hand, by Lemma \ref{lem:split} we have $DS_xM\subset DS_x\Ind^G_KM_i$ for $i=1,2$. We obtain a contradiction.
      \end{proof}

      \begin{proposition}\label{prop:fullness} Assume that $(G,K)$ is a proper pair.
       Let $M_1, M_2$ be indecomposable $G$-modules of non-zero superdimension. Assume there exists an indecomposable $K$-module $V$, $\sdim(V)\neq 0$, such that $\Res_K M_i$ has $V$ as a direct summand for $i=1,2$. Then $M_1 \cong M_2$.
      \end{proposition}
      \begin{proof}
      Assume that $M_1 \not\cong M_2$.
       By Lemma \ref{lem:split}, we have both $M_1$ and $M_2$ as direct summands of $\Ind_K^G V$, with embeddings $\iota_i: M_i\to \Ind_K^G V$, and projections $\pi_i: \Ind_K^G V \to M_i$ for $i=1,2$. Let us show that $M_1 \oplus M_2$ is a direct summand of $\Ind_K^G V$ as well.
       
       Indeed, let $\iota = \iota_1 +\iota_2$ and $\pi=\pi_1+\pi_2$. Then $\pi\circ \iota \in \End_G(M_1 \oplus M_2)$ is $\id + \pi_2\iota_1 + \pi_1\iota_2$. Let us show that this is an isomorphism; in other words, that $\pi_2\iota_1 + \pi_1\iota_2$ is nilpotent. Here $\pi_2\iota_1: M_1 \to M_2$, $\pi_1\iota_2:M_2 \to M_1$. Indeed, $$ (\pi_2\iota_1 + \pi_1\iota_2)^2 = \pi_1\iota_2\pi_2\iota_1 \oplus \pi_2\iota_1\pi_1\iota_2$$ where the first summand is a map $M_1 \to M_1$ (factoring through $M_2$) and the second is a map $M_2 \to M_2$ (factoring through $M_1$). These maps clearly commute, and neither is an isomorphism, since $M_1 \not\cong M_2$ and they are indecomposable. The fact that $M_1, M_2$ are indecomposable implies that the commuting maps $\pi_1\iota_2\pi_2\iota_1$, $\pi_2\iota_1\pi_1\iota_2$ are nilpotent, and hence $\pi_2\iota_1 + \pi_1\iota_2$ is nilpotent.
       
       Thus $M_1 \oplus M_2$ is a direct summand of $\Ind_K^G V$.

       Now, since $V$ is a direct summand of $\Res_K M_i$, we have: $\Ind_K^G V$ is a direct summand of $\Ind_K^G \Res_K M_1$. So $M_1 \oplus M_2$ is a direct summands of $\Ind_K^G \Res_K M_1$. Hence the module
       $$DS_x(M_1)\cong DS_x \Ind_K^G \Res_K M_1$$
       (the isomorphism given by Lemma \ref{lem:ind}) has $DS_x(M_1) \oplus DS_x(M_2)$ as a direct summand, which is only possible if $DS_x(M_2)=0$. Yet the latter contradicts the assumption that $\sdim(M_2) \neq 0$.

      \end{proof}

      \begin{theorem}\label{thm:mainss} If $(G,K)$ is a proper pair then the functor $F:\overline{\Rep G}\to \overline{\Rep K}$ is full.
        \end{theorem}
        \begin{proof} Since $\overline{\Rep G}$ and  $\overline{\Rep K}$ are semisimple categories, fullness is equivalent to the following two properties:
          \begin{enumerate}
          \item If $X$ is simple then $F(X)$ is simple.
            \item If $F(X)\simeq F(Y)$ for two simple $X,Y\in  \overline{\Rep G}$ then $X\simeq Y$.
            \end{enumerate}

            On the other hand, isomorphism classes of simple objects in  $\overline{\Rep G}$ (respectively,  $\overline{\Rep K}$)
            are in bijection with isomorphism classes of indecomposable objects of non-zero superdimension in $\Rep G$ (respectively, $\Rep K$).
            Hence Proposition \ref{prop:full} implies (1) and   Proposition \ref{prop:fullness} implies (2). 
          \end{proof}

          \subsection{Main example} Recall now the theory of Sylow subgroups developed in \cite{SSV}. A minimal splitting subgroup of $G$ is called a Sylow subgroup. It is shown in \cite{SSV} that all Sylow subgroups are conjugate under the action of $G_0$. Further, Theorem 9.14 and the table 9.3 in \cite{SSV} imply the following:
          \begin{proposition}\label{prop:Sylow} Let $G$ be a quasireductive supergroup with Lie superalgebra $\g$. Assume that
            $\g/Z(\g)$ does not contain a simple ideal isomorphic to $\p\s\q(n)$ for $n\geq 3$. Let $S\subset G$ be a Sylow subgroup and $K=N_G(S)$ be the normalizer of $S$. Then $(G,K)$ is a proper pair.
            \end{proposition}

        \subsection{Example for which $F$ is not essentially surjective} Let $G=GL(n|n)$, and consider $K=GL(1|1)^n\subseteq G$.  Then $N=N_G(K)=GL(1|1)^n\rtimes S_n$.  We have $N/K\cong S_n$, and thus $\Rep S_n$ is a full subcategory of $\Rep N$.  

        \begin{proposition}
            If $V$ is an irreducible, nontrivial representation of $S_n$, then $V$ is not in the essential image of $F$.
        \end{proposition}

        \begin{proof}
        Let $u=\begin{bmatrix} 0 & I_n\\ 0 & 0\end{bmatrix}$.  Then $u\in\k_{\overline{1}}^{hom}$, and further $U$ is fixed by the subgroup $H\subseteq G_0$ of invertible matrices $\begin{bmatrix}
            A & 0\\ 0 & A
        \end{bmatrix}$.  In particular our subgroup $S_n\subseteq N$ fixes $u$, and therefore we have a natural action of $S_n$ on $DS_u(M)$ for any $N$-module $M$.  Clearly, $DS_u(V)\cong V$ as an $S_n$-module.  
        
        We have $\Ind_{K}^{G}\C=\Ind_{N}^{G}\C[S_n]$.  Let us suppose that $V$ is in the image of $F$.  Then $V$ must split off $\C[G/K]=\Ind_{K}^{G}\C$ as an $S_n$-module, where $S_n\subseteq N$.  Hence it must also split off $DS_u(\C[G/K])$ as an $S_n$-module.  

        However by Thm.~6.9 of \cite{SS}, $DS_u(\C[G/K])$ is naturally identified, as an $H$-module, with $H_{dR}^\bullet(H/T)$, where $T\subseteq H$ is a maximal torus.  However the $H$-module structure on $H_{dR}^\bullet(H/T)$ is trivial, see for example \cite{ChEi}  In particular, $DS_u(\C[G/K])$ is trivial as an $S_n$-module, giving a contradiction.
        \end{proof}

\section{Defect $1$ supergroups}
\subsection{Defect $1$ supergroups and splitting subgroups}
Let $G$ be a connected affine algebraic supergoup with basic Lie superalgebra $\g$ of defect $1$. Here is the complete list of such supergroups (up to isogeny) : with $n>0,m>0$ we have
$SL(1|n+1)$, $SOSp(2|2n)$, $SOSp(3|2n)$, $SOSp(2m|2)$, $SOSp(2m+1|2)$, $D(2,1;t)$, $AG_2$ ($G_3$ in Kac notations) and $AB_3$ ($F_4$ in Kac notations). \color{red}\begin{footnote} {Some of these supergoups are not simply connected, but the stable categories of a group and its simply connected cover coincide in all our examples.} \end{footnote}\color{black}

Fix a Cartan subalgebra and root decomposition of $\g$. Then a root $\s\l(1|1)$-subalgebra is a Sylow subalgebra of $\g$, \cite{SSV}.
We denote this superalgebra by $\s$ and by $S$ the corresponding algebraic subgroup of $G$. For any non-zero $x\in\s_{\bar 1}$ we have $\g_x:=DS_x\g$ is a quasireductive
Lie superalgebra and $\g_x\simeq \g_y$ for any two $x,y\in\s_{\bar 1}$. By $G_x$ we denote the minimal algebraic subgroup of $G$ with Lie superalgebra
$\g_x$. Denote by $K$ the normalizer of $S$ in $G$. The connected component $K_\circ\subset K$ is isomorphic  to $GL(1|1)\times G_x$ (unless $G=D(2,1;t)$ for $t\notin\mathbb Q$) and
$K=K_\circ$ if $\g$ is of type I, i.e. $\g=\g\l(1|n)$ or $\mathfrak{osp}(2|2n)$; $K/K_\circ\simeq C_2$ in all other cases. Note that the choice of
$G_x$ does not depend on the choice of $x\in \s_{\overline{1}}$.

The table below describes the supergroup $K_\circ$ in all cases except $D(2,1;t)$ for irrational $t$. In the latter case $S$ is a $2|2$-dimensional supergroup, and $K_\circ=\mathbb G_m\ltimes S$ (see Appendix to \cite{SS} for an explicit description of $K_\circ$). The finite central subgroup $\Gamma\subset G_x$ is defined as follows. Let $\alpha$ be an isotropic root, and $x\in \g_\alpha$. Let $P$ be the weight lattice of $G$ and $P_x$ the weight lattice of $G_x$. The lattice
$P_\alpha=\alpha^\perp\cap P_x$ is a sublattice of $P_x$ of finite index and $\Gamma:=P_x/P_\alpha$. The importance of $\Gamma$ will be clear in Theorem \ref{st-def1} below.

\begin{center}
		\begin{tabular}{|c|c|c|}
			\hline 
			$G$ & $G_x$ & $\Gamma$\\
			\hline
			$GL(1|n)$ & $GL(n-1)$ & $\{1\}$\\
                  \hline
                  $SOSp(2|2n)$ & $SP(2n-2)$ & $\{1\}$ \\
			\hline 
			$SOSp(m|2)$ & $SO(m-2)$ & $\{1\}$ \\
			\hline 
			$SOSp(3|2n)$ & $SOSp(1|2n-2)$ & $\{1\}$ \\	
			\hline 
			$D(2,1;\frac{p}{q}), 0<p\leq q, (p,q)=1$ & $\mathbb G_m$ & $\{1\}$ \\ 
			\hline 
			$AG(1|2)$ & $SL(2)$ & $C_2$\\
			\hline 
			$AB(1|3)$ & $SL(3)$ & $C_3$ \\
			\hline 
			\end{tabular}
	\end{center}

        \subsection{On the category $\Rep K$}
Since $S$ is a Sylow subgroup of $K$ we have the following
      \begin{lemma}\label{Kprojectivity} A $K$-module $M$ is projective if and only if the restriction of $M$ to $S$ is projective. \end{lemma}
      \begin{lemma}\label{indecomposable-def1} Assume $G\neq D(2,1;t)$, $t\notin \mathbb Q$.
        Let $M$ be an indecomposable $K_\circ$-module. Then $M=I\boxtimes L$ where $L$ is a simple $G_x$-module and
        $I$ is an indecomposable $GL(1|1)$-module.
      \end{lemma}
      \begin{proof}Note that $G_x$ is a linearly reductive supergroup and $K_\circ=GL(1|1)\times G_x$. Every isotypic $G_x$-component of $M$ splits as a direct summand. Hence $M=I\boxtimes L$
        \end{proof}

        Suppose $\g$ is of type II. Then we have an exact sequence
        $$1\to K_\circ\to K\to C_2\to 1.$$
        This sequence defines an automorphism $\theta\in \operatorname{Aut}K_\circ/\operatorname{Inn}K_\circ$. One can check that $\theta$ 
        maps $\g_\alpha$ to $\g_{-\alpha}$ and induces an outer automorphism of $GL(1|1)\subset K$ (or of $S\subseteq K$ for $D(2,1;t)$ with $t\notin\mathbb Q$).
        \begin{corollary}\label{type2} Assume $G\neq D(2,1;t)$, $t\notin \mathbb Q$ and $G$ is of type II. Then an indecomposable $K$-module $M$ is one of the following
          \begin{enumerate}
          \item $M\simeq \Ind^K_{K_\circ}(I\boxtimes L)$ where either $I^\theta$ is not isomorphic to $I$ or $L^\theta$ is not isomorphic to $L$.
          \item $M$ is isomorphic to one of the two indecomposable components of $\Ind^K_{K_\circ}(I\boxtimes L)$ where $I^\theta\simeq I$ and
            $L^\theta\simeq L$.
            \end{enumerate}
          \end{corollary}
          Consider the contravariant functor $(-)^{\vee}:\Rep G\to\Rep G$ of contragredient duality, and note that it descends to $\Rep K\to\Rep K$.
          \begin{proposition}\label{imp-aux} Let $M$ be an indecomposable $K$-module and $\sdim M\neq 0$. If $M^\vee\simeq M$ then $M$ is simple.
          \end{proposition}
          \begin{proof} First we assume that $G\neq D(2,1;t)$, $t\notin \mathbb Q$. Suppose $M$ is not simple. From Lemma \ref{indecomposable-def1}, Corollary \ref{type2} and the description of indecomposable $GL(1|1)$-modules (see, for instance, \cite{ESS}) we have that  $M$
          has Loewy length $2$ and $\dim \operatorname{soc} M\neq \dim\operatorname{cosoc}M$ since $\sdim M=\pm(\dim \operatorname{soc}M-\dim\operatorname{cosoc}M)$. Since $\sdim\operatorname{soc} M^\vee= \sdim\operatorname{cosoc}M$ we get a contradiction.

            Now let us consider the case $G=D(2,1;t)$ with $t\notin \mathbb Q$. If $\sdim M\neq 0$ then the center of $K$ acts trivially on $M$.
            Since $K/Z(K)\simeq C_2\ltimes PGL(1|1)$ we can repeat all above arguments.
          \end{proof}
          We now describe the structure of the principal block $\Rep^0K$  of $\Rep K$. It follows immediately from the above description of indecomposable modules in $\Rep K$  that $G_x$ acts trivially on all modules in $\Rep^0K$. Therefore, we need to describe the principal blocks of $K/G_x$, which is isomorphic to $PGL(1|1)$ for $\g$ of type I and $C_2\ltimes PGL(1|1)$ for $\g$ of type II.
      \begin{proposition}\label{principalK} 
        \begin{enumerate}
        \item For $\g$ of type I  any  simple module in  $\Rep^0K$ is the one-dimensional $PGL(1|1)$-module $T_n$ with weight $n\in\mathbb{Z}$.
          We have $\sdim T_n=(-1)^n$. 
        \item Let $\g$ be of type II. There are two $1|0$-dimensional simple modules in  $\Rep^0K$-- $\mathbb C$ and $U$, with $U$ being the sign module
          of $C_2$ with trivial action of $PGL(1|1)$. All other simple modules are two-dimensional and isomorphic to $\Ind^K_{K_\circ}T_n$ for $n>0$.
        \end{enumerate}        
      \end{proposition}
      \begin{remark} Although we don't use it we would like to mention that for $\g$ of type
      I the principal block $\Rep^0K$ is equivalent to the category of finite dimensional
      representation of the quiver $A_\infty$ 
      $$\cdots\rightleftarrows\bullet\rightleftarrows\bullet\rightleftarrows \cdots$$
      with the relations $xy=yx$, $x^2=y^2=0$, see \cite{J}. Here we denote by $x$ and $y$ left and right arrows, respectively.

      If $\g$ has type II one can show that $\Rep^0 K$ is equivalent to the category of finite-dimensional representations of $D_\infty$
       $$\begin{matrix} \bullet&{} &{}&{}&{}\\
       \downarrow\uparrow &{} &{}&{}&{}\\
       \bullet &\rightleftarrows &\bullet&\rightleftarrows&\cdots\\
       \downarrow\uparrow &{} &{}&{}&{}\\
       \bullet&{} &{}&{}&{}\\
      \end{matrix}$$
       with the similar relations.
      \end{remark}
      
      \begin{proof} (1) Follows from the identity $$PGL(1|1)/[PGL(1|1),PGL(1|1)]=\mathbb G_m.$$
      (2) A simple module in $\Rep^0 K$ is a direct summand in $\Ind^K_{K_\circ}T_n$, the latter is irreducibe unless $n=0$. For $n=0$ we have $\Ind^K_{K_\circ}\triv=\triv\oplus U$.
    \end{proof}
      \begin{corollary}\label{stable-principalK} The stable category $\St^0K$ of $\Rep^0K$ is generated, as a tensor triangulated category, by $\mathbb C$.
        \end{corollary}
        \begin{proof} When $\g$ is of type I, we consider the exact triangle
        $$\triv\to \Omega^{-1}\triv\to T_1\oplus T_{-1}\to \Omega\triv,$$
        where we write $\Omega$ for the shift.
        Then we see that $T_{\pm 1}$ lie in the category generated by $\mathbb C$ and hence $T_n$ for all $n$.
    Let $\g$ be of type II. The exact triangle 
         $$\triv\to \Omega^{-1}\triv\to \Ind^K_{K_\circ}T_1\to \Omega\triv$$
         implies that $\Ind^K_{K_\circ}T_1$ lies in the category generated by $\triv$. 
         Using $$\Ind^K_{K_\circ}T_m\otimes \Ind^K_{K_\circ}T_1=\Ind^K_{K_\circ}T_{m-1}\oplus\Ind^K_{K_\circ}T_{m+1}$$ for $m>1$ and
         $$(\Ind^K_{K_\circ}T_1)^{\otimes 2}=\Ind^K_{K_\circ}T_2\oplus U\oplus\mathbb C$$
         we see that all simple modules  lie in the category generated by $\triv$.
        \end{proof}
               \subsection{The restriction functor}
        
          \begin{theorem}\label{restriction} 
            (a)  Let $L$ be a simple non-projective $G$-module. Then $\Res_KL\simeq \dot{L}\oplus P$ for some simple non-projective $K$-module $\dot{L}$ and some projective $K$-module $P$.

            (b) If $L$ and $M$ are simple non-projective $G$-modules and $\dot{L}\simeq\dot{M}$ then $L\simeq M$.
          \end{theorem}
          \begin{proof} (a) Proposition \ref{prop:full} implies that $\Res_KL=\dot{L}\oplus \bigoplus_i M_i$ where all $M_i$ are indecomposable $K$-modules of superdimension zero and $\dot{L}$ is an indecomposable module of nonzero superdimension. Theorem 1.3 in \cite{ESS} implies that $M_i$ restricted to $GL(1|1)$ does not contain non-projective summands of superdimension zero. Lemma \ref{indecomposable-def1} and Corollary \ref{type2} imply that every $M_i$ is a projective $K$-module. It remains to show that $\dot{L}$ is simple. But $\dot{L}$ is indecomposable and $\dot{L}^\vee\simeq\dot{L}$. Thus, simplicity of $\dot{L}$ follows from Proposition \ref{imp-aux}.
            
(b) follows from Proposition \ref{prop:fullness}
\end{proof}
\begin{lemma}\label{image} For any simple $G$-module $L$ the subgroup $\Gamma$ acts trivially on $\dot{L}$.  
\end{lemma}
\begin{proof} Let $\mu\subset P$ be some weight of $\dot{L}$. Then $(\mu,\alpha)=0$ since $[\g_{\alpha},\g_{-\alpha}]$ acts by zero on $\dot{L}$.
  Therefore $\mu\in\alpha^\perp\cap P$. This implies the statement.
\end{proof}
\subsection{The induction functor}
For a finite dimensional $G$-module $M$ we denote by $\dot{M}$ a minimal $K$-submodule of $M$ such that $\Res_KM=\dot M\oplus P$ where $P$ is a projective $K$-module.
As we know from Theorem \ref{restriction}, if $M$ is simple then $\dot{M}$ is simple. 
\begin{proposition}\label{ind-proj} For every finite dimensional $G$-module $M$ we have $\Ind^G_KM=(\Ind^G_K\dot{M})\oplus Q$ for some projective
 $G$-module $Q$.
\end{proposition}
\begin{proof} Indeed, $Q=\Ind^G_K P$ is projective since $P$ is projective.
\end{proof}
\begin{proposition}\label{ind-triv} We have $\Ind^G_K\mathbb C=\mathbb C\oplus Q$ for some projective $G$-module $Q$.
\end{proposition}
\begin{proof} Since $K$ is a splitting subgroup we have  $\Ind^G_K\mathbb C=\mathbb C\oplus Q$ for some $G$-module $Q$. It remains to show that $Q$ is projective.  Let $L$ be a simple submodule of $Q$ which is not projective and not trivial. Let $\tilde\varphi$ denote the embedding $L\to\Ind^G_K\mathbb C$. It suffices to show that $\tilde\varphi$ is a composition $L\to I(L)\to \Ind^G_K \triv$ where $L\to I(L)$ is the injective envelope.

By Frobenius reciprocity $\tilde\varphi$ is induced by some
  $\varphi\in \Hom_K(L,\mathbb C)$. Consider $\varphi^*\in \Hom_K(\mathbb C,L^*)$ and the induced map
  $\widetilde{\varphi^*}:\triv\to \Ind^G_KL^*$.
  Since $\dot{L^*}$ is not trivial we have $\Hom_G(\mathbb C,\Ind^G_K\dot{L}^*)=0$ so by Proposition \ref{ind-proj}, the image of $\widetilde{\varphi^*}$ lies in a projective (and hence injective) summand of $ \Ind^G_KL^*$. That means that $\widetilde{\varphi^*}$ is a composition of two morphisms
  $$\mathbb C\to I(\mathbb C)\to \Ind^G_KL^*.$$
  Note that $\tilde{\varphi}$ is the composition of maps:
  $$L\to I(\mathbb C)\otimes L\to \Ind^G_KL^*\otimes L\to \Ind^G_K\triv,$$
  where the last arrow is induced by the coevaluation $L\otimes L^*\to\triv$. Since $\sdim L\neq 0$
  $\Ind^G_K\triv$ is a direct summand in $\Ind^G_KL^*\otimes L$. Thus, we obtain that $\tilde\varphi$ is a composition: $L\to I(L)\to\Ind^G_K\mathbb C$. 
\end{proof}

\begin{corollary}\label{cor-ind} For any $M\in \Rep G$ we have $\Ind^G_K\dot{M}=M\oplus Q$ for some projective $G$-module $Q$.
\end{corollary}
\begin{proof} By Proposition \ref{ind-triv} we have
  $\Ind^G_KM=M\oplus P$ for some projective $P$. By Proposition \ref{ind-proj}, $\Ind^G_K\dot{M}\oplus R=\Ind^G_KM$ for some projective $R$. The statement follows.
\end{proof}
\subsection{Equivalence of stable categories} Note the the restriction functor $\Res_K:\Rep G\to\Rep K$ maps projective modules to projective modules. Therefore it induces a symmetric monoidal functor of tensor triangulated categories $\Res _K:\St G\to\St K$. Lemma \ref{image} implies that the latter functor descends to
$\St G\to\St K/\Gamma$.
 \begin{theorem}\label{st-def1} The functor $\Res_K$ defines an equivalence of tensor triangulated categories $\St G$ and $\St K/\Gamma$.
\end{theorem}
\begin{proof} We first prove the following statement. 
\begin{lemma}\label{essurj}$\Res_K$ is essentially surjective. 
\end{lemma}
\begin{proof} Recall by Corollary \ref{stable-principalK} that $\St^0K$ is generated by $\C$. Therefore
  $\Res_K:\St^0 G\to\St^0 K$ is essentially surjective. If $G=D(2,1;t)$ for $t\notin\mathbb Q$ then $\St G=\St^0G$ and $\St K=\St^0K$ and the statement follows. Let us consider the cases with many atypical blocks.
  
  Let $K'=K/(\Gamma\times GL(1|1))$. Note that $K'$ is linearly reductive. Let $L$ be a faithful $K'$-module. By abuse of notation, we denote by the same letter the pull-back of $L$ in $\Rep K$.  We claim that if $L=\Res_K(M)$ for some $M\in St G$ then $\Res_K$ is essentially surjective.
  It follows from the fact that $L$ generates $\Rep K'$ and hence $\St^0 K$ and $L$ generate
  $\St K$. 

  Consider first the case where $G$ is one of the classical supergroups with defining representation $V$.
  Then $L=\Res_K V$ satisfies the desired property.

  For $G=AG_2$ or $AB_3$ one can set $L=\Res_K \g$. 

  Finally for $G=D(2,1;\frac{p}{q})$
  with $1<p<q$ consider the minimal simple atypical representation $V$ which does not lie in the principal block. The highest weight of $V$ with respect to the Borel subalgebra with simple roots
 $-\varepsilon_1+\varepsilon_2+\varepsilon_3,\varepsilon_1-\varepsilon_2+\varepsilon_3,\varepsilon_1+\varepsilon_2-\varepsilon_3$ equals
 $p\varepsilon_1+q\varepsilon_2$. One can check that $\dot{L}$ is a 2-dimensional faithful representation of $K'$.
\end{proof}
 Now Corollary \ref{cor-ind} implies that $\Ind^G_K$ induces a well defined inverse functor $\St K\to\St G$.
\end{proof}
\begin{corollary}\label{tensprodG} If $M,N$ are simple $G$-modules then $M\otimes N$ is a direct sum of semisimple and projective modules.
             \end{corollary}
  \subsection{Semisimplification} Recall the functor $F$ from Theorem \ref{thm:mainss}. 
  \begin{theorem}\label{ss-def1} If $G$ is a defect $1$ supergroup and $K$ the normalizer of its Sylow subgroup, then the functor $F:\overline{\Rep G}\to \overline{\Rep K/\Gamma}$ is an equivalence of categories.
  \end{theorem}  
  \begin{proof} Theorem \ref{st-def1} implies that $F$ defines a bijection between isomorphism classes of indecomposable $G$-modules and indecomposable
    $K$-modules of nonzero superdimension.
    \end{proof}
    Recall the result of T. Heidensdorf, \cite{H}, that $$\overline{\Rep GL(1|1)}=\overline{\Rep PGL(1|1)}\simeq \Rep \mathbb G_m^2. $$ 
    Using the above theorem we can compute semisimplification for all defect $1$ groups. In the table below we use the notation $\bar G$ for the linearly reductive supergroup such that $\overline{\Rep G}\simeq \Rep\bar G$.  In the semidirect product $C_2\ltimes \mathbb G_m^2$ the generator of $C_2$ acts by
                      permutation of components in $\mathbb G_m\times\mathbb G_m$ and by involutive outer automorphism on $SO(2r)$ and $PSL(3)$.
 In the case of  $D(2,1;t)$ for rational $t$ the generator of $C_2$ acts on $\mathbb G_m^3$ by permuting two generators and inverting the third generator.
	\renewcommand{\arraystretch}{1.5}	
 
\begin{center}
		\begin{tabular}{|c|c|}
			\hline 
			$G$ & $\bar G$\\
			\hline
			$GL(1|n)$ & $\mathbb G_m^2\times GL(n-1)$ \\
                  \hline
                  $SOSp(2|2n)$ & $\mathbb G_m^2\times SP(2n-2)$\\
			\hline 
			$SOSp(2m+1|2), $ & $(C_2\ltimes \mathbb G_m^2)\times SO(2m-1))$ \\
                  \hline
                  	$SOSp(2m|2), m\geq 2$ & $C_2\ltimes (\mathbb G_m^2\times SO(2m-2))$ \\
			\hline 
			$SOSp(3|2n)$ & $(C_2\ltimes \mathbb G_m^2)\times SOSp(1|2n-2)$ \\	
			\hline 
			$D(2,1;\frac{p}{q}), 0<p<q, (p,q)=1$ & $(C_2\ltimes \mathbb G_m^3)$ \\ 
                  \hline
                  	$D(2,1;t),t\notin \mathbb Q$ & $C_2\ltimes \mathbb G_m^2$ \\ 
			\hline 
			$AG(1|2)$ & $(C_2\ltimes \mathbb G_m^2)\times PSL(2)$ \\
			\hline 
			$AB(1|3)$ & $C_2\ltimes (\mathbb G_m^2\times PSL(3))$ \\
			\hline 
			\end{tabular}
                      \end{center}

                      \subsection{Defect $1$ blocks for higher rank supergroups} Let $G=GL(m|n)$ or $SOSp(m|2n)$. One can find classification of blocks
                      of $\Rep G$ in \cite{GS}, Theorem 2. Every block $\mathcal B$ is equivalent as an abelian category to the principle blocks in $\Rep G_{\mathcal B}$, where $G_{\mathcal B}=GL(k|k)$,
$SOSp(2k|2k)$,  $SOSp(2k+2|2k)$ or $SOSp(2k+1|2k)$. The number $k$ is called degree of atypicality or defect of
the block. We denote by  $K_{\mathcal B}$ the normalizer of a Sylow subgroup of $G_{\mathcal B}$. Then Theorem \ref{st-def1} implies the following
\begin{corollary}\label{cor-def1-gen} The stable categories of a block $\mathcal B$ of defect $1$ and the principal block of $\St K_{\mathcal B}$ are equivalent.
  \end{corollary}

\section{Defect group of the block}
\subsection{Thick ideals generated by simple modules}
In this section, we assume that $G=GL(m|n)$ or $OSp(m|2n)$ (not connected). In this case, the Sylow subgroup $S\subset G$ is isomorphic to $SL(1|1)^d$ where $d$ is a defect of $G$.
More precisely, $d=\min (m,n)$ for $GL(m|n)$ and $\min(\lfloor{\frac{m}{2}}\rfloor,n)$ for
$OSp(m|2n)$.
It is shown in \cite{PSS} that $$\Y_G:=\g_1^{hom}/(G_0\times\mathbb G_m)$$ is homeomorphic to the Balmer spectrum of
the stable category $\St G$. Recall that we denote by $K$ the normalizer of $S$ in $G$.

Consider a map 
$$p: \s_{\bar 1}/(K_0\times \mathbb G_m)\to \Y_G$$ defined by
$p(x)=G_0\cdot x$. One can check that p induces a bijection of sets. The topology on $\Y_G$ is inherited from the induced topology on $\s_{\bar 1}$.

For every $k\leq d$ we consider the natural embedding $S_k:=SL(1|1)^k\subset S$ and define
$$\Y^k_G:=p((\s_k)_{\bar 1}/\mathbb G_m).$$
It is clear that $\Y^k_G$ is a closed subset of $\Y_G$ and hence there exists a unique radical thick ideal
$\cI_k$ of $\St G$ consisting of all $M\in\St G$ such that $\supp M\subseteq\Y^k_G$.
We say that $\rk(x)=k$ if $k$ is the minimal number such that $G_0\cdot x\cap\s_k\neq\emptyset$.  

\begin{lemma}\label{strat-supp} Let $M$ be a simple $G$-module of atypicality degree $k$,
then $\supp M=\Y^k_G$.
\end{lemma}
\begin{proof} We have to show that $DS_x(M)\neq 0$ if and only if $\rk(x)\leq k$. First, Theorem 2.1 in \cite{S} implies that $DS_x M=0$ if $\rk(x)>k$ and $[x,x]=0$. By closeness of support, this implies $DS_x M=0$ for any homological $x$ with $\rk(x)>k$.

Now let $\rk(x)=k$. Assume first that $[x,x]=0$. Theorem 12.1(3) in \cite{GHSS} implies that  $DS_x M$ is a nonzero semisimple 
typical $\g_x$-module satisfying the purity property $[DS_xM:L][DS_xM:\Pi L]=0$. Now let $y$ be any homological element of rank $k$. We may assume that $y,x\in\s_k$. Then $\g_y=\g_x$ and both $x,y$ are morphisms
of $\g_x$-modules $M\to\Pi M$. The purity property implies that $[\Res_{\g_x}M:L]\neq [\Res_{\g_x}M:\Pi L]$
for some simple typical $\g_x$-module $L$. The latter implies $DS_y M\neq 0$.
Since $\supp M$ is closed, and the closure of $\Y_G^k\setminus\Y_G^{k-1}$ coincides with $\Y_G^k$, the statement follows.
\end{proof}
\begin{corollary}\label{cor-supp} The ideal $\cI_k$ is generated by any simple module of atypicality degree $k$. 
\end{corollary}
\begin{proof} Let $\cI(M)$ denote the ideal generated by a simple module $M$ of atypicality degree $k$. Due to Lemma \ref{strat-supp} we just have to show that $\cI(M)$ is the radical ideal.
 For this we have to check that $M$ lies in $\cI(M\otimes M)$, which is clear since $M$ is a direct summand of $M\otimes M\otimes M^*$. In fact, in a rigid tensor category thick ideals are always radical.   
\end{proof}
 
\subsection{Relative projectivity} 

Let $H\subset G$ be a pair of quasireductive supergroups. We call a $G$-module $M$ {\it $H$-projective} if any of the following equivalent conditions holds:
\begin{enumerate}
\item if an exact sequence $0\to N\to R\to M\to 0$ splits over $H$ then it splits over $G$;
\item the restriction map $\Ext_G^*(R,M)\to \Ext_K^*(R,M)$ is injective for any $G$-module $R$;
\item there exists an $H$-module $V$ such that $M$ is a direct summand of the induced module
$\Ind^G_H V$;
\item the natural embedding $M\to\Ind^G_H M$ splits.
\end{enumerate}
The equivalence of the above conditions can be proven in the same way as for modular representations of finite groups.

The following proposition explains some relation between support and  relative projectivity.
\begin{proposition}\label{bigatyp} Assume that $M$ is $S_k$-projective. Then $M\in\cI_k$.    \end{proposition}
\begin{proof} If $M$ is relatively projective with respect to $S_k$ then $M$ splits as a direct summand in $\Ind^G_{S_k}M$. Suppose that $x\in\g_{\bar 1}^{hom}$ and $\rk(x)>k$. Then the vector field induced by $x$ does not have fixed points on $G/S_k$ and therefore $DS_x\Ind^G_{S_k}M=0$, see \cite{SS}. This implies $DS_xM=0$. Therefore $\supp M\subset \mathbb Y_G^k$.
    \end{proof}

\subsection{Defect group of a block} 
Here we again assume that $G=GL(m|n)$ or $SOSp(m|2n)$.
 Recall the classification of blocks, see \cite{GS} (Section 5) and Section 2.7.
 Note that the Sylow subgroup of $G_{\mathcal B}$ is isomorphic to $S_k$. 
Let  $\mathcal B_0$ denote the principal block in $\Rep G_{{\mathcal B}}$.
We denote by $F: \mathcal B\to\mathcal B_0$ the functor which establishes equivalence of blocks.

\begin{theorem}\label{defectblock} Let $\mathcal B$ be a block in $\Rep G$ of atypicality degree $k$.
Then every $M\in\mathcal B$ is $S_k$-projective.    
\end{theorem}
\begin{proof} The functor $F^*$ inverse to $F$ is the left adjoint to $F$. It is not difficult to see that 
$$F^*(V)=p_{\mathcal B}(\Ind^G_{G_\mathcal B} V) ,$$
where $p_{\mathcal B}$ denotes the projection to the block $\mathcal B$.
Since $F^*F(M)=M$, we obtain that $M$ is a direct summand in $\Ind^G_{ G_\mathcal B} F(M)$ and hence $M$ is $G_{\mathcal B}$-projective.
But $S_k$ is a splitting subgroup of $G_{\mathcal B}$ and hence any $G_{\mathcal B}$-module is $S_k$-projective. Thus, $M$ is $S_k$-projective.
\end{proof}

 Theorem \ref{defectblock}, Proposition \ref{bigatyp} and Corollary \ref{cor-supp} imply the following. 
\begin{corollary} Given a block $\mathcal B$ of atypicality $k$. Then $S_k$ is a minimal subgroup of $S$ with property that all modules in $\mathcal B$  are $S_k$-projective. 
\end{corollary}

\subsection{Open questions}
\begin{enumerate}
    \item Is converse of Proposition \ref{bigatyp} true? Is any module in $\mathcal I_k$ also
    $S_k$-projective?
    \item Let us call the defect group $D$ of a block $\mathcal B$ a minimal subgroup of $G$ with property that every module in $\mathcal B$ is $D$-projective. Are all defect groups conjugate?
    \item By Theorem \ref{thm:mainss} we do have a surjection $\bar K\to\bar G$ of reductive envelopes of $G$ and $K$ respectively. Describe the kernel of this surjection.  Equivalently, the essential image of $\Rep \bar G$ in $\Rep\bar K$ under $F$.
\end{enumerate}

\bibliographystyle{amsalpha}

\end{document}